\documentclass[11pt, a4paper]{amsart}
\usepackage{amsmath,amssymb,amsfonts, float}
\usepackage[all]{xy}
\usepackage{caption}
\usepackage{enumerate}
\usepackage{mathpazo}
\usepackage{a4wide}
\usepackage{diagbox}
\usepackage{empheq}

\usepackage{bm}
\usepackage{graphicx}
\usepackage[breaklinks=true]{hyperref}

\usepackage{xcolor}
\usepackage{multicol}
\usepackage{tabularx}

\usepackage{hyperref}
\usepackage[capitalize]{cleveref}

\usepackage[normalem]{ulem}

\newtheorem{thm}{Theorem}[section] 
\newtheorem*{thm*}{Theorem} 
\newtheorem{prop}[thm]{Proposition}
\newtheorem{lem}[thm]{Lemma}
\newtheorem{cor}[thm]{Corollary}

\theoremstyle{definition}
\newtheorem{definition}[thm]{Definition}
\newtheorem{expl}[thm]{Example}

\newtheorem{rem}[thm]{Remark}

\DeclareMathOperator{\Z}{\mathbb{Z}}

\DeclareMathOperator{\F}{\mathbb{F}}

\DeclareMathOperator{\Div}{{\rm Div}}

\DeclareMathOperator{\Tr}{{\rm Tr}}

    \DeclareFontFamily{U}{wncy}{}
    \DeclareFontShape{U}{wncy}{m}{n}{<->wncyr10}{}
    \DeclareSymbolFont{mcy}{U}{wncy}{m}{n}
    \DeclareMathSymbol{\Sha}{\mathord}{mcy}{"58}

\numberwithin{equation}{section}

\DeclareSymbolFont{bbold}{U}{bbold}{m}{n}
\DeclareSymbolFontAlphabet{\mathbbold}{bbold}

\usepackage{hyperref}

\newcommand{\Rad}{{\rm Rad}}
\newcommand{\s}{\text{ss}}

\newcommand{\rad}{{\rm rad}}

\begin{document}
\title{Isomorphic gcd-graphs over polynomial rings}
 \author{ J\'an Min\'a\v{c}, Tung T. Nguyen, Nguy$\tilde{\text{\^{e}}}$n Duy T\^{a}n }
\address{Department of Mathematics, Western University, London, Ontario, Canada N6A 5B7}
\email{minac@uwo.ca}
\dedicatory{Dedicated to Professor Ki-Bong Nam on the occasion of his 70th birthday }

 \address{Department of Mathematics and Computer Science, Lake Forest College, Lake Forest, Illinois, USA}
 \email{tnguyen@lakeforest.edu}
 
  \address{
Faculty Mathematics and 	Informatics, Hanoi University of Science and Technology, 1 Dai Co Viet Road, Hanoi, Vietnam } 
\email{tan.nguyenduy@hust.edu.vn}

\thanks{JM is partially supported by the Natural Sciences and Engineering Research Council of Canada (NSERC) grant R0370A01. He gratefully acknowledges the Western University Faculty of Science Distinguished Professorship 2020-2021. He is also grateful for the support of the Western Academy for Advanced Research in 2022-2023 and the current support of the Fields Institute for research in Mathematical Sciences.  TTN is partially supported by an AMS-Simons Travel Grant. NDT is partially supported by the Vietnam National Foundation for Science and Technology Development (NAFOSTED) under grant number 101.04-2023.21}
\keywords{Gcd-graphs, Function fields, Integral graphs, Isospectral graphs.}
\subjclass[2020]{Primary 05C25, 05C50, 05C51}
\maketitle

\begin{abstract}
Gcd-graphs over the ring of integers modulo $n$ are a simple and elegant class of integral graphs. The study of these graphs connects multiple areas of mathematics, including graph theory, number theory, and ring theory. In a recent work, inspired by the analogy between number fields and function fields, we define and study gcd-graphs over polynomial rings with coefficients in finite fields. We discover that, in both cases, gcd-graphs share many similar and analogous properties. In this article, we extend this line of research further. Among other topics, we explore an analog of a conjecture of So and a weaker version of Sander-Sander concerning the conditions under which two gcd-graphs are isomorphic or isospectral. We also provide several constructions showing that, unlike the case over $\mathbb{Z}$, it is not uncommon for two gcd-graphs over polynomial rings to be isomorphic. 
\end{abstract}

\tableofcontents

\section{Introduction}

From the time when Ren\'e Descartes introduced his Cartesian coordinate system in 1637, the interplay between geometry and algebra has become one of the most powerful tools in mathematics. This interplay is particularly powerful in algebraic graph theory, where it is natural to associate graphs with their adjacency matrices and study the spectra of these matrices, leading to the rich theory of graph spectra.

There are also other natural ways to associate graphs with algebraic structures. In fact, graphs defined over algebraic structures, particularly groups, have been extensively studied since the pioneering work of Cayley, who introduced what are now known as Cayley graphs (see \cite{cayley1878desiderata}). These graphs provide a natural framework for visualizing abstract algebraic structures: the elements of the underlying structure are represented by vertices, while a prescribed generating set determines the edges between them. Consequently, many graph-theoretic properties can be described in terms of the algebraic properties of the underlying structure. Conversely, various algebraic properties can be realized by the geometry of the associated graph. For example, a subset of a group generates the whole group if and only if the corresponding Cayley graph is connected. This simple but useful observation has been exploited in several settings. For instance, the authors of \cite{maimani2010rings} use connectivity properties of associated graphs to classify finite commutative rings that are generated by their units.

The study of Cayley graphs and related constructions also draws upon several areas of mathematics, including group theory, representation theory, Fourier analysis on finite groups, matrix theory, and algebraic geometry. Indeed, natural symmetries of groups and rings are reflected in the symmetries of their associated graphs. Conversely, it is often interesting to investigate how much of the underlying algebraic structure can be recovered from its associated graph. When the underlying algebraic structure is a finite ring, the associated graphs become particularly interesting because both the additive and multiplicative structures of the ring come into play. One important direction in this area is the study of gcd-graphs, developed in the foundational work of Klotz and Sander (see for example, \cite{unitary,anderson2021graphs,klotz2007some,nguyengcd2026}). Unlike Cayley graphs considered solely from the perspective of the additive group, graphs defined over rings can incorporate information from both ring operations. The interaction between the additive and multiplicative structures often gives rise to rich algebraic, combinatorial, and spectral properties.

Since the work of Klotz and Sander, the literature has seen an explosion of research exploring many further fundamental properties of these graphs, including their connectedness, bipartiteness, perfectness, clique and independence numbers, spectral properties, and much more. We refer readers to  \cite{unitary,bavsic2015polynomials, chudnovsky2024prime,kiani2012unitary} and the references therein for further discussions around this line of research. 

We first recall the definition of a gcd-graph.
\begin{definition}
    Let $A$ be a principal ideal domain and $n \in A$ which is not a unit. Suppose further that the ring $A/n$ is finite. Let $\Div(n)$ be the set of divisors of $n$ (defined up to associates) and $D \subset \Div(n)$ such that $n \not \in D.$ The gcd-graph $G_{n}(D)$ 
is the graph equipped with the following data: 
    \begin{enumerate}
    \item The vertex set of $G_n(D)$ is $A/n.$
    \item Two vertices $a,b \in A/n$ are adjacent if and only if $\gcd(a-b, n) \in D.$
\end{enumerate}
In other words, $G_n(D)$ is the Cayley graph on $A/n$ with the generating set 
\[ S_D = \{h \in A/n  \mid \gcd(h,n) \in D \}.\] 
\end{definition}
We remark that a gcd-graph is necessarily simple and undirected; that is, $a$ is adjacent to $b$ if and only if $b$ is adjacent to $a$. This follows from the identity
\[
\gcd(a-b,n)=\gcd(b-a,n).
\]
Throughout the paper, we assume that all graphs are simple and undirected.

The case $A = \Z$ is discussed in \cite{klotz2007some} and the case $A = \F_q[x]$ is the main topic of \cite{minavc2024gcd}.  By their own nature, we can see that the study of these gcd-graphs bridges several branches of mathematics including graph theory, number theory, commutative algebra, and character theory for finite groups. For example, when $A = \Z$, using the theory of Ramanujan sums, the authors of \cite{klotz2007some} show that gcd-graphs are integral, meaning all of their eigenvalues are integers. In \cite{so2006integral}, So shows that the converse is true as well: a $\Z/n$-circulant graph is integral if and only if it is a gcd-graph. So also poses a conjecture about whether a gcd-graph $G_{n}(D)$ determined $D$ and $n$ (see \cite[Conjecture 7.3]{so2006integral}). While this conjecture is still open, some progress has been made. For example, Sander and Sander in \cite{sander2015so} ask that for a given $n$ whether $D$ is determined by the spectral vector $(\sum_{d \in D} c(\ell, \frac{n}{d}))_{\ell = 1}^n$ where $c(\ell, \frac{n}{d})$ is the Ramanujan sum (see \cite[Section 4]{klotz2007some} for the definition). We remark that the components of this spectral vector are exactly the eigenvalues of $G_{n}(D)$ counted with multiplicity. In the same paper, Sander and Sander prove this conjecture. In \cite{schlage2021determinant}, Schlage-Puchta gives a new and shorter proof for the weak conjecture of Sander-Sander using a determinant involving Ramanujan sums.

Given what is already known in the case $A= \Z$, one may ask whether the conjecture of So and the weaker version of Sander-Sander hold in the case $A= \F_q[x].$ As we will see in this article, the answer is negative for the first question (though it fails for a good reason, see \cref{rem:same_degree})  and affirmative for the second question. In fact, for the first question, we will provide various constructions of isomorphic gcd-graphs for different $f$ and $D.$ One particular reason for this stark difference between the number case and the function field case is that over function fields, we can find $f \neq g$ such that $\F_q[x]/f \cong \F_q[x]/g$ whereas in the number field case, this is impossible. For the second question, we will show that the approach of \cite{schlage2021determinant} can be adapted naturally in the function field setting, leading to a proof of the weak conjecture of Sander-Sander in this case. 
\subsection{Outline.}
In \cref{sec:spectral_vector}, we prove the analog of the weaker conjecture of Sander-Sander in the function fields setting. More precisely, we show that for a fixed $f \in \F_q[x]$, $D$ is determined by a spectral vector describing all eigenvalues of $G_{f}(D)$. We achieve this by studying a matrix involving Ramanujan sums whose counterpart over $\Z$ was introduced in \cite{schlage2021determinant}. In \cref{sec:prime_powers-I}, we study the graph theoretic properties of gcd-graphs $G_{f}(D)$ where $f$ is a prime power and in \cref{sec:prime_powers-II} we study their spectral properties. 
In both sections, we show that many of the results concerning the spectrum of $G_{f}(D)$, as described in \cite{sander2018structural, sander2015so}, have analogs in the function field setting. Furthermore, we explore several fundamental graph-theoretic properties of $G_{f}(D)$ such as their bipartiteness, perfectness, clique, and independence numbers, which, to the best of our knowledge, have not yet been addressed in the literature—even for gcd-graphs over $\Z$. \cref{sec:isomorphic_gcd} studies isomorphism between gcd-graphs. This is where we will see differences between gcd-graphs over $\Z$ and gcd-graphs over $\F_q[x].$ By analyzing experimental data via the Python library NetworkX, we provide several constructions of isomorphic gcd-graphs. 
\subsection{Code}
Many insights in this paper are gained through an extensive analysis of experimental data. The code that we wrote to generate data and do experiments with it can be found at  \cite{Nguyen_isomorphic_gcd_graph}.  Additionally, we have verified all statements in this work with various concrete and computable examples.

\section{Spectral vector determines $D$} \label{sec:spectral_vector}
Let $q$ be a prime power, and $f \in \F_q[x]$ a monic polynomial. As we explain in \cite[Section 6]{minavc2024gcd}, there is a direct analogy between the character theory of $\Z/n$ and that of $\F_q[x]/f$. More specifically, while the character theory of $\Z/n$ can be fully described once we fix a primitive $n$-th root of unity, the character theory of $\F_q[x]/f$ is similarly determined by a fixed non-degenerate functional on $\F_q[x]/f$. As a result, the spectra of gcd-graphs in both cases have an explicit description via Ramanujan sums. For the case of gcd-graphs over $\Z$, we refer readers to \cite[Section 4]{klotz2007some}. For $\F_q[x]$, we will now recall the definition of the Ramanujan sums which describe the spectra of gcd-graphs.

\begin{definition} (see \cite[Section 6.2]{minavc2024gcd})
Let $f \in \F_q[x]$ be a monic polynomial. For each $g \in \F_q[x]$, the Ramanujan sum $c(g,f)$ is defined as follows. 
 \[ c(g,f) = c_{\psi}(g,f) = \sum_{a \in (\F_q[x]/f)^{\times}} \zeta_p^{\Tr(\psi(ga))}. \]
 Here $\psi: \F_q[x]/f \to \F_q$ is a non-degenerate functional on $\F_q[x]/f$ and $\Tr: \F_q \to \F_p$ is the trace map. 
\end{definition}
While the above definition looks rather complicated, we can show that $c(g,f)$ has a simple expression almost analogous to the case of Ramanujan sums over $\Z$. In particular, $c(g,f)$  does not depend on the choice of $\psi$ as long as we make sure that it is non-degenerate. More precisely, by \cite[Section 6.2]{minavc2024gcd} we have 
\begin{equation} \label{eq:formula_for_c}
c(g,f)  =  \mu(t) \dfrac{\varphi(f)}{\varphi(t)}, \quad \text{where} \quad t = \dfrac{f}{\gcd(f,g)}. 
\end{equation}

Here $\varphi(f)=|(\F_q[x]/f)^{\times}|$ is the analog of the classical Euler function and $\mu(t)$ is the analog of the classical Mobius function (see \cite{rosen2013number} for some further discussions about these functions).

The following definition is inspired by the work \cite{sander2015so} of Sander-Sander on gcd-graphs over $\Z.$
\begin{definition}
Let $\Div(f)$ be the set of monic divisors of $f$.    Let $D \subset \Div(f)$ such that $f \not \in D$. The spectral vector of the gcd-graph $G_{f}(D)$ is defined to be the following vector 
    \[ \vec{\lambda}(f,D) = \left(\lambda_{g}(f,D) \right)_{g \in \F_q[x]/f} \] 
    where 
    \begin{equation}  \label{eq:definition_of_lambda}
    \lambda_{g} (f,D) = \sum_{d \in D} c\left(g, \frac{f}{d} \right).
    \end{equation}
\end{definition}

By the main result in \cite[Section 6]{minavc2024gcd}, the components of this spectral vector are precisely the eigenvalues of $G_{f}(D)$ counted with multiplicity. The following theorem is a direct analog of \cite[Theorem 1.2]{sander2015so}.

\begin{thm} \label{thm:same_spectral_vector}
    Let $D_1, D_2$ be two proper subsets of $Div(f).$ Suppose that $G_{f}(D_1)$ and $G_{f}(D_2)$ have the same spectral vector. Then $D_1 = D_2.$
\end{thm}
To prove \cref{thm:same_spectral_vector}, we adapt the strategy employed in \cite{schlage2021determinant} which deals with gcd-graphs over $\Z$. More precisely, we will introduce and study a matrix similar to the one defined in \cite{schlage2021determinant}.  Let $g,h \in \F_q[x]$ and $c(g,h)$ the Ramanujan sum. Let $C_f = (c(g,h))_{g,h \in \Div(f)}.$  In \cite[Theorem 1]{schlage2021determinant}, the author solves the weak conjecture of Sander-Sander by showing that the determinant of $C_n$ is not zero. In the case of $\F_q[x]$, we have a similar statement.

\begin{prop} 
Let $|f|$ be the norm of $f$; i.e., $|f|$ is the order of the finite ring $\F_q[x]/f.$ Then 
\[ \det(C_f)= \pm |f|^{\frac{\tau(f)}{2}} .\] 
Here $\tau(f)$ is the number of monic divisors of $f.$ In particular, $\det(C_f) \neq 0.$
\end{prop}
We remark that the sign of $\det(C_f)$ depends on the ordering of $\Div(f).$ For our application, this sign is not important. 

\begin{proof} We proceed by induction on the number of monic irreducible factors of $f$. If $\deg f=0$, then $f=1$ and $\det(C_1)=1$. Now let $\deg f>0$ and $P$ a monic irreducible polynomial coprime to $f$. Let $n$ be a positive integer. Let $f_1,f_2,\ldots, f_r$ be the monic divisors of $f$, here $r=\tau(f)$. We list the monic divisors of $P^nf$ as
\[
f_1,f_2,\ldots,f_r,Pf_1,Pf_2,\ldots,Pf_r, \ldots,P^2f_1,P^2f_2,\ldots,P^2f_r,\ldots, P^nf_1,P^nf_2,\ldots,P^nf_r.
\]
Then $C_{P^nf}$ is an $(n+1)\times(n+1)$-block matrix
\[
\begin{bmatrix}
C_{00}&C_{01}&\cdots&C_{0n}\\
C_{10}&C_{11}&\cdots&C_{1n}\\
\vdots&\vdots&\cdots&\vdots\\
C_{n0}&C_{n1}&\cdots&C_{nn}
\end{bmatrix},
\]
where $C_{ij}$ is a $r\times r$-matrix whose $(k,l)$-entry is
\[
(C_{ij})_{kl}=c(P^id_k,P^jd_l).
\]
Clearly $C_{00}=C_f$.   If $j\geq i+2$ then
\[
c(P^id_k,P^jd_l)=\dfrac{\varphi(P^jd_l)}{\varphi\left(\dfrac{P^jd_l}{\gcd(P^id_k,P^jd_l)}\right)}\mu\left(\dfrac{P^jd_l}{\gcd(P^id_k,P^jd_l)}\right)=0.
\]
This is because $\dfrac{P^jd_l}{\gcd(P^id_k,P^jd_l)}$ is divisible by $P^2$.

If $j=i+1$ then
\[
\begin{aligned}
c(P^id_k,P^jd_l)&=\dfrac{\varphi(P^jd_l)}{\varphi\left(\dfrac{P^jd_l}{\gcd(P^id_k,P^jd_l)}\right)}\mu\left(\dfrac{P^jd_l}{\gcd(P^id_k,P^jd_l)}\right)\\
&=\dfrac{\varphi(P^{i+1})}{\varphi(P)} \dfrac{\varphi(d_l)}{\varphi\left(\dfrac{d_l}{\gcd(d_k,d_l)}\right)}\mu(P)\mu\left(\dfrac{d_l}{\gcd(d_k,d_l)}\right)\\
&=-|P|^i \dfrac{\varphi(d_l)}{\varphi\left(\dfrac{d_l}{\gcd(d_k,d_l)}\right)}\mu\left(\dfrac{d_l}{\gcd(d_k,d_l)}\right)=-|P|^ic(d_k,d_l),
\end{aligned}
\]
and hence $C_{i,i+1}=-|P|^iC_f.$ 

If $j\leq i$ then 
\[
\begin{aligned}
c(P^id_k,P^jd_l)&=\dfrac{\varphi(P^jd_l)}{\varphi\left(\dfrac{P^jd_l}{\gcd(P^id_k,P^jd_l)}\right)}\mu\left(\dfrac{P^jd_l}{\gcd(P^id_k,P^jd_l)}\right)\\
&=\dfrac{\varphi(P^j)\varphi(d_l)}{\varphi\left(\dfrac{d_l}{\gcd(d_k,d_l)}\right)}
\mu\left(\dfrac{d_l}{\gcd(d_k,d_l)}\right)
=\varphi(P^j)c(d_k,d_l),
\end{aligned}
\]
and hence $C_{ij}=\varphi(P^j)C_f$. Thus
\[
C_{P^nf}= \begin{bmatrix}
C_f&-C_f&0&0&\cdots&0\\
C_f&\varphi(P)C_f&-|P|C_f&0&\cdots&0\\
C_f&\varphi(P)C_f&\varphi(P^2)C_f&-|P|^2C_f&\cdots&0\\
\vdots&\vdots&\cdots&\vdots\\
C_f&\varphi(P)C_f&\varphi(P^2)C_f&\varphi(P^3)C_f&\cdots&-|P|^{n-1}C_f\\
C_f&\varphi(P)C_f&\varphi(P^2)C_f&\varphi(P^3)C_f&\cdots&\varphi(P^n)C_f
\end{bmatrix}.
\]
Subtracting the $n$th row from the $(n+1)$-st row and noting that $\varphi(P^n)+|P|^{n-1}=|P^n|$, we get
\[
\begin{aligned}
\det (C_{P^nf})&=\pm \det\begin{bmatrix}
C_f&-C_f&0&0&\cdots&0\\
C_f&\varphi(P)C_f&-|P|C_f&0&\cdots&0\\
C_f&\varphi(P)C_f&\varphi(P^2)C_f&-|P|^2C_f&\cdots&0\\
\vdots&\vdots&\cdots&\vdots\\
C_f&\varphi(P)C_f&\varphi(P^2)C_f&\varphi(P^3)C_f&\cdots&-|P|^{n-1}C_f\\
0&0&0&0&\cdots&|P|^nC_f&
\end{bmatrix}.
\end{aligned}\]
Expanding along the last block row gives
\[
\begin{aligned}
\det (C_{P^nf})&=\pm |P|^{rn}\det(C_f)\det\begin{bmatrix}
C_f&-C_f&0&0&\cdots&0\\
C_f&\varphi(P)C_f&-|P|C_f&0&\cdots&0\\
C_f&\varphi(P)C_f&\varphi(P^2)C_f&-|P|^2C_f&\cdots&0\\
\vdots&\vdots&\cdots&\vdots\\
C_f&\varphi(P)C_f&\varphi(P^2)C_f&\varphi(P^3)C_f&\cdots&\varphi(P^{n-1})C_f
\end{bmatrix}.
\end{aligned}
\]
Repeating this argument, we obtain
\[
\begin{aligned}    
\det (C_{P^nf}) &=\pm |P|^{rn}\det(C_f) |P|^{r(n-1)}\det(C_f)\cdots|P^r|\det(C_f)\\
&=\pm |P|^{\frac{n(n+1)}{2}\tau(f)}(\det C_f)^{n+1}.
\end{aligned}
\]
By the induction hypothesis
\[
\det (C_{P^nf})=\pm |P|^{\frac{n(n+1)}{2}\tau(f)}(|f|^{\frac{\tau(f)}{2}})^{n+1}=\pm(|P|^n|f|)^{\frac{(n+1)\tau(f)}{2}}=\pm|P^nf|^{\frac{\tau(P^nf)}{2}},
\]
and we are done.
\end{proof}

We can now give a proof of \cref{thm:same_spectral_vector}. 

\begin{proof}
    Assume that $G_{f}(D_1)$ and $G_{f}(D_2)$ have the same spectral vector $\vec{\lambda}(f,D_1) = \vec{\lambda}(f, D_2)$. Then, the subvectors $(\lambda_{g} (f,D_1))_{g \mid f}$ and $(\lambda_{g} (f,D_2))_{g \mid f}$ indexed by divisors of $f$ are also equal. For each $i \in \{1, 2\}$, we define the following  indicator vector of size $|\tau(f)| \times 1$ 
\[
  (v_i)_{h} = 
\begin{cases}
    1 & \qquad\text{if } f/h \in D_i, \\
    0 & \qquad \text{if } f/h \not \in D_i. \\
\end{cases}\]         
We then see that the vectors $(\lambda_{g} (f,D_1))_{g \mid f}$ (respectively $(\lambda_{g} (f,D_2))_{g \mid f}$) are precisely $C_f v_1$ (respectively $C_f v_2$). Since $C_f$ is invertible, we conclude that $v_1 = v_2$ and therefore, $D_1 = D_2.$

\end{proof}
\section{Gcd-graphs associated with prime powers I: \\ Graph theoretic properties of $G_{P^k}(D)$} \label{sec:prime_powers-I}

In this section, we study the gcd-graphs $G_{f}(D)$ in the case where $f$ is a prime power; i.e., $f=P^k$ where $P \in \F_q[x]$ is an irreducible polynomial and $k \geq 1.$  In this case, a subset $D$ of $\Div(P^k)$ can be written uniquely in the following form $D = \{P^{k_1}, P^{k_2}, \ldots, P^{k_s}\}$ where 
\[ 0 \leq k_1 <k_2 < \cdots < k_s <k. \] 
We will fix this notation throughout this section. 

We first discuss some fundamental graph-theoretic properties of $G_{P^k}(D)$. Recall that for each $f \in \F_q[x]$, $|f| = q^{\deg(f)}$ is the order of the finite ring $\F_q[x]/f$.

\begin{prop} \label{prop:connected_prime_power}
    $G_{P^k}(D)$ has exactly $|P|^{k_1}$ connected components, and each component is isomorphic to $G_{P^{k-k_1}}(D')$ where 
    \[ D' = \{1, P^{k_2-k_1}, \ldots, P^{k_s-k_1} \} .\]
    In particular, $G_{P^k}(D)$ is connected if and only if $k_1=0.$
\end{prop}

\begin{proof}
    By definition, the cosets $\{g + P^{k_1} (\F_q[x]/P^k)\}$ where $g$ runs over $\F_q[x]/P^{k_1}$ are mutually unconnected. Furthermore, by \cite[Lemma 5.4]{minavc2024gcd}, each of these cosets is isomorphic to $G_{P^{k-k_1}}(D')$ which is connected by \cite[Corollary 3.4]{minavc2024gcd}.
\end{proof}

\begin{prop}
   Suppose that $G_{P^k}(D)$ is connected. Then, $G_{P^k}(D)$ is a bipartite graph if and only if the following conditions hold. 

    \begin{enumerate}
        \item $\F_q = \F_2.$
        \item $\deg(P)=1$; namely either $P=x$ or $P=x+1.$
        \item $D = \{1 \}.$
    \end{enumerate}
\end{prop}

\begin{proof}
    Suppose that $G_{P^k}(D)$ is a bipartite graph. By \cite[Corollary 4.2]{minavc2024gcd}, we must have $\F_q=\F_2$ and $\gcd(P^k, x(x+1)) \neq 1.$ Since $P$ is irreducible, we conclude that either $P=x$ or $P=x+1.$ By \cite[Theorem 4.3]{minavc2024gcd}, we also know that for each $1 \leq i \leq s$, $\gcd(P^k, x(x+1)) \nmid P^{k_i}.$ This happens only if $k_i=0$; or equivalently $D= \{1 \}.$ Conversely, if all of the above conditions are satisfied, then by \cite[Theorem 4.3]{minavc2024gcd} $G_{P^k}(D)$ is bipartite. In this case, we can in fact show a concrete bipartite partition of $G_{P^k}(D)$ as follows: 
    \[ V(G_{P^k}(D)) = A_0 \bigsqcup A_1, \]
    where 
    \[ A_0 = \{ h \in \F_q[x]/P^k \text{ such that } P \mid h \}, \]
    and 
     \[ A_1 = \{ h \in \F_q[x]/P^k \text{ such that } P \nmid h \}. \qedhere \]
\end{proof}

We now discuss the decomposition of $G_{P^k}(D)$ into the wreath product (also known as the lexicographic product) of simpler graphs. First, we recall the definition of the wreath product and homogeneous sets (see \cite[Section 2]{hammack2011handbook}).

\begin{definition} \label{def:wreath_product}
  Let $\Gamma,\Delta$ be two graphs. The wreath product  of $\Gamma $ and $ \Delta$ is the graph $\Gamma * \Delta$ equipped with the following data 
  \begin{enumerate}
      \item The vertex set of $\Gamma * \Delta$ is the Cartesian product $V(\Gamma ) \times V(\Delta )$,
      \item $(x,y)$ and $(x',y')$ are adjacent in $\Gamma * \Delta$ if either $(x,x') \in E(\Gamma)$ or $x=x'$ and $(y,y') \in E(\Delta)$. 
  \end{enumerate}
\end{definition}

\begin{definition}
Let $G$ be a graph. A homogeneous set in  $G$ is a set $X$ of vertices of $G$ such that every vertex in $V(G) \setminus X$ is adjacent to either all or none of the vertices in $X$. A homogeneous set $X$ is said to be non-trivial if $2 \leq |X| < |V(G)|$.
\end{definition}

As shown in \cite[Section 3]{chudnovsky2024prime}, the existence of a homogeneous set in a Cayley graph is almost equivalent to the existence of a decomposition of $G$ into a wreath product. In \cite[Section 5]{minavc2024gcd}, we describe the necessary and sufficient conditions for the existence of homogeneous sets in the gcd-graphs $G_{f}(D)$ over $\F_q[x].$ When $f=P^k$, the situation is relatively simple. In fact, by \cite[Theorem 5.5]{minavc2024gcd}, we have the following. 

\begin{thm} \label{thm:wreath_product_prime_power}
Suppose that $k \geq 2.$    Let $I = \langle P^{k-1} \rangle$ be the ideal of $\F_q[x]/P^k$ generated by $P^{k-1}$. Then $I$ is a homogeneous set in $G_{P^k}(D).$ Furthermore 

\[
G_{P^k}(D) \cong 
\begin{cases}
    G_{P^{k-1}}(D_1) * K_{|P|} & \qquad\text{if } P^{k-1} \in D, \\
    G_{P^{k-1}}(D_1) * E_{|P|} & \qquad \text{if } P^{k-1} \not \in D. \\
\end{cases}\] 
Here $D_1 = D \setminus \{P^{k-1}\}$, $*$ is the wreath product, $K_m$ (respectively $E_m$) is the complete graph (respectively the co-complete graph) on $m$ nodes. 
\end{thm}

We discuss various consequences of \cref{thm:wreath_product_prime_power}. 
First, let us briefly recall the definitions of graph concepts. Let $G$ be a graph. The clique number $\omega(G)$ of $G$ is the maximum number of vertices in a complete subgraph of $G$. The chromatic number $\chi(G)$ is the minimum number of colors needed to color the vertices of $G$ so that adjacent vertices receive different colors. The independence number $\alpha(G)$ is the maximum number of vertices in an independent set of $G$; that is, a set of vertices no two of which are adjacent. We refer readers to \cite[Section 2.3]{hammack2011handbook} for some further discussions and results regarding these numbers. Graph $G$ is said to be perfect if for every induced subgraph of $G$, the clique number equals the chromatic number.

It is known that the wreath product $G_1 * G_2$ is perfect if and only if both $G_1, G_2$ are perfect (see \cite{lovasz1972normal}). Since complete and co-complete graphs are perfect, by mathematical induction, we conclude from \cref{thm:wreath_product_prime_power} that

\begin{cor}
    $G_{P^k}(D)$ is a perfect graph. 
\end{cor}
Next, we discuss the clique, chromatic, and independent numbers of $G_{P^k}(D)$. 
By \cite{hammack2011handbook}, we know that the clique number $\omega(G)$ of a graph $G$ behaves well with respect to wreath products: $ \omega(G_1 * G_2) = \omega(G_1) \omega(G_2)$. Therefore, we have the following conclusion.

\begin{cor}
    The clique and chromatic numbers of $G_{P^k}(D)$ are $|P|^{|D|} = |P|^s.$
\end{cor}
\begin{proof}
By \cref{thm:wreath_product_prime_power}, we know that 
\[
G_{P^k}(D) \cong 
\begin{cases}
    G_{P^{k-1}}(D_1) * K_{|P|} & \qquad \text{if } P^{k-1} \in D, \\
    G_{P^{k-1}}(D_1) * E_{|P|} & \qquad \text{if } P^{k-1} \not \in D. \\
\end{cases}\] 
Here $D_1 = D \setminus \{P^{k-1}\}$. Since the clique number behaves well with respect to the wreath product, we have 

\[
\omega(G_{P^k}(D)) =  
\begin{cases}
    |P| \omega(G_{P^{k-1}}(D_1)) & \qquad\text{if } P^{k-1} \in D, \\
    \omega(G_{P^{k-1}}(D_1)) & \qquad \text{if } P^{k-1} \not \in D. \\
\end{cases}\] 
By induction, we conclude that $\omega(G_{P^k}(D))=|P|^{s}.$ Since $G_{P^k}(D)$ is perfect, its chromatic number is also equal to $|P|^s.$
\end{proof}

Finally, the wreath product also behaves well with respect to taking complements: $(G_1 * G_2)^c = (G_1)^c * (G_2)^c$ (see \cite[Page 44]{hammack2011handbook}). Therefore, by an identical argument, we also have: 

\begin{cor}
    The independence number of $G_{P^k}(D)$ is $|P|^{k-|D|} = |P|^{k-s}.$
\end{cor}

\section{Gcd-graphs associated with prime powers II: \\ Spectral properties of $G_{P^k}(D)$} \label{sec:prime_powers-II}
We now focus on the spectral properties of $G_{P^k}(D)$. While several of the results in this section are similar to those in \cite{sander2018structural,sander2015so}, many are new, thanks to insights inspired by the experimental data that we generate for this project.

\subsection{Spectral descriptions for $G_{P^k}(D)$.}
We first have the following observation about Ramanujan sums. 

\begin{lem} \label{lem:two_powers}
Suppose that $P \in \F_q[x]$ is a monic irreducible polynomial. Let $m,k$ be two non-negative integers. Then 
\[
c(P^m, P^k) =
\begin{cases}
    \varphi(P^k) & \qquad \text{if } m \geq k, \\
    -|P|^m & \qquad \text{if } k-m=1, \\
    0 & \qquad\text{if } k- m \geq 2.
\end{cases}
\]
\end{lem}

\begin{proof}
    This follows directly from \cref{eq:formula_for_c} and from the formula $\varphi(P^e)=|P|^{e-1}(|P|-1).$
\end{proof}

Let us introduce the following notation which is inspired by a counterpart over $\Z$ (see \cite[Theorem 1.1]{sander2015so}).
\begin{definition}
We define the following function 
    \[ \chi(P^k, D, t) = \begin{cases}
    1 &  \qquad\text{if } k_i = k -t -1 \text{ for some } 1 \leq i \leq s, \\
    0 &  \qquad\text{else.}
    \end{cases}\]
\end{definition}

With this notation, we can now calculate the spectrum of $G_{P^k}(D)$ explicitly. The following statement is a direct analog of \cite[Theorem 1.1]{sander2015so} in the function fields setting. 
\begin{thm} \label{prop:spectra_prime_power}

Let $g \in \F_q[x]/P^k$ and $P^t = \gcd(P^k, g).$  Let $\lambda_{g}(P^k,D)$ be the eigenvalue of $G_{P^k}(D)$ as described in \cref{eq:definition_of_lambda}. Then we have the following. 
\begin{align*} \lambda_{g}(P^k, D) &= \lambda_{P^t}(P^k, D) = -\chi(P^k, D, t)|P|^t + \sum_{k_i \geq k-t} \varphi(P^{k-k_i}) \\
&= -\chi(P^k, D, t)|P|^t + (|P|-1) \sum_{k_i \geq k-t} |P|^{k-k_i-1}.
\end{align*}
\end{thm}

\begin{proof}
By \cref{eq:definition_of_lambda}, we know that  $c(g,f) = c(\gcd(f,g),f)$ for all $f$ and $g$. In particular,  we have $\lambda_{g}(P^k, D) = \lambda_{P^t}(P^k, D)$. We then have 
\[  \lambda_{P^t}(P^k, D) = \sum_{i=1}^s c(P^t, P^{k-k_i}).\]
By \cref{lem:two_powers} and the definition of $\chi(P^k,D,t)$ we conclude that 
\[ \lambda_{P^t}(P^k, D) = -\chi(P^k, D, t)|P|^t + \sum_{k_i \geq k-t} \varphi(P^{k-k_i}) = -\chi(P^k, D, t)|P|^t + (|P|-1) \sum_{k_i \geq k-t} |P|^{k-k_i-1}, \] 
and we are done.
    \end{proof}

We now use \cref{prop:spectra_prime_power} to describe various arithmetic and algebraic properties of the spectrum of $G_{P^k}(D).$

\subsection{Congruence of eigenvalues of $G_{P^k}(D)$}
In this subsection, by utilizing \cref{prop:spectra_prime_power}, we investigate some congruence properties of eigenvalues of $G_{P^k}(D)$. We then give an explicit criterion for a number $m$ to be an eigenvalue of $G_{P^k}(D)$ for some choice of $D$. In particular, we show that every integer $m$ can be realized as an eigenvalue of $G_{P^k}(D)$ for appropriate choices of $P$ and $D.$ We start this section with an observation which is suggested by our numerical data. 

\begin{prop} \label{prop:congruence}
Let $\lambda \in \Z$ be an eigenvalue of $G_{P^k}(D).$ Then 
\[
\lambda \equiv 
\begin{cases}
    0 \pmod{|P|-1} & \qquad \text{if } \lambda \geq 0, \\
    -1 \pmod{|P|-1} & \qquad \text{if } \lambda <0. 
\end{cases} \]     
\end{prop}
\begin{proof}
    
Suppose that $\lambda = \lambda_g(P^k, D)$ for some $g$, where we keep the same notation as in \cref{prop:spectra_prime_power}. We have \[ \lambda_{g}(P^k, D) = \lambda_{P^t}(P^k, D) = -\chi(P^k, D, t)|P|^t + \sum_{k_i \geq k-t} \varphi(P^{k-k_i}). \] 
    We have 
    \[ 0 \leq  \sum_{k>k_i \geq k-t} \varphi(P^{k-k_i}) \leq \sum_{i=1}^t \varphi(P^i) = (|P|-1) \frac{|P|^{t}-1}{|P|-1} = |P|^{t}-1.\]
Additionally, since $\chi(P^k, D, t) \in \{0, 1 \}$, we conclude that $\lambda \geq 0$ if and only if $\chi(P^k, D, t) = 0$. Therefore 
\begin{equation} \label{eq:formula_for_lambda}
\lambda = 
\begin{cases}
    \sum_{k_i \geq k-t} \varphi(P^{k-k_i}) = (|P|-1) \sum_{k_i \geq k-t} |P|^{k-k_i-1} & \text{if } \lambda \geq 0, \\
    -|P|^t + \sum_{k_i \geq k-t} \varphi(P^{k-k_i}) =-|P|^t + (|P|-1) \sum_{k_i \geq k-t} |P|^{k-k_i-1} & \text{if } \lambda <0. \\
\end{cases}
\end{equation}

We then conclude that 
\begin{empheq}[left={\lambda \equiv \empheqlbrace}]{align*}
  0 \quad & \pmod{|P|-1} && \text{if } \lambda \geq 0, \\
  -1 \quad & \pmod{|P|-1} && \text{if } \lambda < 0. 
\end{empheq}
This completes the proof.
\end{proof}
We have a direct corollary of \cref{prop:congruence}. 
\begin{cor} \label{cor:possible_eigen}
    Let $\lambda$ be an integer satisfying the congruence condition 
    \[
\lambda \equiv 
\begin{cases}
    0 \pmod{|P|-1} & \qquad\text{if } \lambda \geq 0, \\
    -1 \pmod{|P|-1} & \qquad\text{if } \lambda <0. \\
\end{cases} \]     
    
    Let $\lambda_0$ be a non-negative integer defined by the following rule 
    \[ \lambda_0 = \begin{cases}
    \dfrac{\lambda}{|P|-1} & \qquad \text{if } \lambda \geq 0, \\
    \dfrac{|\lambda+1|}{|P|-1} & \qquad\text{if } \lambda <0. \\
\end{cases} \]     
Then $\lambda$ is an eigenvalue of $G_{P^k}(D)$ for some $k$ and $D$ if and only if each digit in the base-$|P|$ representation of $\lambda_0$ is either $0$ or $1.$
\end{cor}

\begin{proof}
Let us keep the same notation as in the proof of \cref{prop:spectra_prime_power}. Then, we have 
\[
\lambda = 
\begin{cases}
    (|P|-1) \sum_{k_i \geq k-t} |P|^{k-k_i-1} & \qquad\text{if } \lambda \geq 0, \\
    -|P|^t + (|P|-1) \sum_{k_i \geq k-t} |P|^{k-k_i-1} & \qquad \text{if } \lambda <0. \\
\end{cases}\]       
Consequently, we have 
\[
\lambda_0 = 
\begin{cases}
    \sum_{k_i \geq k-t} |P|^{k-k_i-1} & \qquad\text{if } \lambda \geq 0,\\
    \sum_{i=0}^{t-1} |P|^i - \sum_{k_i \geq k-t} |P|^{k-k_i-1} & \qquad\text{if } \lambda <0. \\
\end{cases} \]       
We can see that the statement follows directly from this formula.
\end{proof}

If $|P|=2$ then we can see that the congruence condition mentioned in \cref{prop:congruence} is trivial. As a result, we have the following. 
\begin{prop}
    For each integer $m$, there exists an integer $k$ and a subset $D \subset \Div(P^k)$ with $P=x \in \F_2[x]$ such that $m$ is an eigenvalue of $G_{P^k}(D).$
\end{prop}

\subsubsection{\textbf{Some special eigenvalues of $G_{P^k}(D)$}}
We now discuss the existence of small eigenvalues. Specifically, for each $\lambda \in \{-1, 0, 1\}$, we find the necessary and sufficient conditions for $\lambda$ to be an eigenvalue of $G_{P^k}(D).$ Quite surprisingly, we will see that $0$ and $-1$ cannot simultaneously be eigenvalues of $G_{P^k}(D)$ (see \cref{cor:exactly_one}). We start our discussion with the eigenvalue $0.$
\begin{prop} \label{prop:eigen_zero}
The following statements are equivalent.    
\begin{enumerate}
        \item $P^{k-1} \not \in D.$

        \item $0$ is an eigenvalue of $G_{P^k}(D).$
    \end{enumerate}
Consequently, there are exactly $2^{k-1}$ graphs in the family $G_{P^k}(D)$ that have $0$ as an eigenvalue. 
\end{prop}

\begin{proof}
    Let us first show that $(1)$ implies $(2).$ Let $\lambda = \lambda_{P^0}(P^k, D).$ Since $P^{k-1} \not \in D$ we know that $\chi(P^k, D, 0)=0$. Consequently, by \cref{prop:spectra_prime_power}, we conclude that $\lambda =0.$

    Conversely, let us assume that $\lambda = \lambda_{P^t}(P^k, D)= 0$ is an eigenvalue of $G_{P^k}(D).$ Suppose to the contrary that $P^{k-1} \in D$. By the proof of \cref{prop:congruence} we know that $\chi(P^k, D, t)=0$ and 
    \[ \sum_{k_i \geq k-t} |P|^{k-k_i-1}=0.\]
    Since $P^{k-1} \in D$, this implies that $t=0.$ However, this would imply that $\chi(P^k, D, 0)=1$ which is a contradiction. We conclude that $P^{k-1} \not \in D.$
\end{proof}
For the eigenvalue $-1$, we have the following proposition. 
\begin{prop} \label{prop:eigen_negative_one}
The following statements are equivalent. 
\begin{enumerate}
    \item $P^{k-1} \in D.$

    \item $-1$ is an eigenvalue of $G_{P^k}(D).$
\end{enumerate}
Consequently, there are exactly $2^{k-1}$ graphs in the family $G_{P^k}(D)$ that have $-1$ as an eigenvalue. 
\end{prop}
\begin{proof}
    Let us first show that $(1)$ implies $(2).$ Since $P^{k-1} \in D$, we conclude that $\chi(P^k, D, 0) =1.$ By \cref{prop:spectra_prime_power}, we conclude that $\lambda_{P^0}(P^k, D)=-1.$

    Conversely, suppose that $\lambda = \lambda_{P^t}(P^k, D)=-1$ is an eigenvalue of $G_{P^k}(D).$ Then $\chi(P^k, D, t)=1$ and 
     $-1= \lambda = -|P|^t + (|P|-1) \sum_{k_i \geq k-t} |P|^{k-k_i-1}.$ This implies that 
    \[ \sum_{i=0}^{t-1} |P|^{i}= \sum_{k_i \geq k-t} |P|^{k-k_i-1}. \]
    We conclude that $P^{k-1} \in D.$
\end{proof}

By combining \cref{prop:eigen_negative_one} and \cref{prop:eigen_zero} we have the following rather surprising corollary. 

\begin{cor} \label{cor:exactly_one}
    For each $D$, exactly one of the values $\lambda =0$ or $\lambda =-1$ is an eigenvalue of $G_{P^k}(D).$
\end{cor}

Our numerical data suggest that $1$ is an eigenvalue of $G_{P^k}(D)$ under some very special conditions. We remark that by \cref{prop:connected_prime_power}, we only need to consider the case  where $G_{P^k}(D)$ is connected; i.e., $1 \in D.$

\begin{prop}
Suppose that $1 \in D$. Then $1$ is an eigenvalue of $G_{P^k}(D)$ if and only if the following conditions hold. 
\begin{enumerate}
    \item $\F_q= \F_2.$
    \item $\deg(P)=1.$
    \item $P^{k-1} \in D.$
    \item $P^{k-2} \not \in D.$
\end{enumerate}
We remark that when $k=1$, the  condition (4) is understood as vacuous. 

\end{prop}

\begin{proof}
First, let us assume that $\lambda =1$ is an eigenvalue of $G_{P^k}(D)$. Then the congruence relation described in \cref{prop:congruence} implies that $|P|-1$ is a divisor of $\lambda.$ This happens only if $2=|P|=q^{\deg(P)}.$ We conclude that $q=2$ and $\deg(P)=1.$ Furthermore, by \cref{eq:formula_for_lambda}, we can find $t$ such that $\chi(P^k, D, t)=0$ and 
\[ 1= \sum_{k_i \geq k-t} |P|^{k-k_i-1}.\]
The summands are distinct powers of $2$, and therefore there is only one summand equal to $2^0=1$. Hence, $P^{k-1}\in D$.  If $t=1$, then $\chi(P^k,D,1)$ gives $ P^{k-2}\notin D.$
If $t\geq 2$ and $P^{k-2}\in D$, then $ k-2\geq k-t$, so the term corresponding to $P^{k-2}$ occurs in the displayed sum. and it contributes
$2^{k-(k-2)-1}=2$,
which is a contradiction to the value of the sum being $1$. Therefore, $P^{k-2}\notin D$ in all cases. For the other direction, we can see that if all of the above conditions are satisfied, then $\lambda_{P^1}(P^k,D)=1.$
\end{proof}

\subsubsection{\textbf{Some estimates for the eigenvalues of $G_{P^k}(D)$}}

In this subsubsection, we discuss some estimates on the eigenvalues of $G_{P^k}(D).$ Here, we provide a direct analog of \cite[Section 3]{sander2015so}. Keeping the notation as in \cref{prop:spectra_prime_power}, we have the following (compare with \cite[Corollary 2.1]{sander2015so}).

\begin{prop} \label{prop:estimate}
Let $0 \leq u \leq v <k$. Then 
\[ |\lambda_{P^u}(P^k, D)| \leq |\lambda_{P^v}(P^k, D)|. \]

Furthermore, $\lambda_{P^k}(P^k, D)$ is the degree of $G_{P^{k}}(D)$, which is also its largest eigenvalue. 
\end{prop}

\begin{proof}
    This statement follows directly from \cref{eq:formula_for_lambda} and the fact that for each $t \geq 1$, we have 
    \[ |P|^t > (|P|-1) \sum_{i=0}^{t-1} |P|^i.
    \qedhere\]
\end{proof}

\begin{rem}
    The proof for \cite[Corollary 2.1]{sander2015so} is somewhat more complicated because the authors allow loops in $G_{P^k}(D)$, i.e., they consider the possibility that $P^k \in D.$
\end{rem}
We now discuss some corollaries of \cref{prop:estimate}. The first corollary concerns the largest eigenvalue in $G_{f}(D).$

\begin{cor} \label{cor:largest_eigenvalues}
    Let $D_1, D_2$ be two subsets of $\Div(P^k)$ such that $G_{P^{k}}(D_1)$ and $G_{P^k}(D_2)$ have the same largest eigenvalue. Then $D_1 = D_2.$
\end{cor}
\begin{proof}
    Let $D_1=\{P^{k_1},P^{k_2}, \ldots, P^{k_s}\}$ and $D_2 = \{P^{h_1}, P^{h_2}, \ldots, P^{h_t}\}$. Suppose that $G_{P^k}(D_1)$ and $G_{P^k}(D_2)$ have the same largest eigenvalue. Then  $\sum_{i=1}^s |P|^{k-k_i} = \sum_{i=1}^t |P|^{k-h_i}.$ Let $u$ be the last index such that $k_u \neq h_u$. Then we have 
    \[ \sum_{i=u}^s |P|^{k-k_i} = \sum_{i=u}^t |P|^{k-h_i}.\]
If $u=s$ or $u=t$, then we must have $s=t$ and $D_1 = D_2.$ Otherwise, suppose that $u < \min(s,t)$. Without loss of generality, we can also assume that $k-k_{u}>k-h_u.$ We have then 
\[ \sum_{i=u}^s |P|^{k-k_i} \geq |P|^{k-k_u} \geq |P|^{k-h_u+1} \geq \sum_{i=0}^{k-h_u}|P|^i \geq \sum_{i=u}^t |P|^{k-h_i}. \]
We conclude that $D_1= D_2.$
\end{proof}

\begin{rem}
    Similar to the case of gcd-graphs over $\Z$ (see \cite[Section 1]{sander2015so} for an example in this case), \cref{cor:largest_eigenvalues} is false if $f$ is not a prime power. For example, let $f=x^2(x+1) \in \F_3[x]$, $D_1 = \{1\}$ and $D_2 = \{x, x+1, x(x+1) \}.$ Then the characteristic polynomials of $G_{f}(D_1)$ and $G_{f}(D_2)$ are respectively 
    \[ (x - 12) (x - 3)^4 (x + 6)^4 x^{18}, \]
    and 
    \[ (x - 12) (x - 6)^2 (x - 3)^2  (x + 3)^{10}  x^{12} .\]
    We see that $G_{f}(D_1)$ and $G_{f}(D_2)$ have the same largest eigenvalue, which is 12, but they are not isomorphic since they do not have the same spectrum.
\end{rem}

We discuss another corollary of \cref{prop:estimate}, which we found by investigating our numerical data. 
\begin{cor} \label{cor:at_least_two_eigs}
    Suppose that $D \neq \emptyset$. Then $G_{P^k}(D)$ has at least two distinct eigenvalues. Furthermore, $G_{P^k}(D)$ has exactly two eigenvalues if and only if 
    \[ (k_1, k_2, \ldots, k_s) = (k_1, k_1+1, \ldots, k_1+s-1). \]
    In this case, $G_{P^k}(D)$ is the disjoint union of $|P|^{k_1}$ copies of $K_{m}$ where $m=|P|^{k-k_1}$ and its eigenvalues are $-1$ and $|P|^{k-k_1}-1.$
\end{cor}

\begin{proof}
We know that the largest eigenvalue of $G_{P^k}(D)$ is its degree which is 
\[ \lambda_{P^k}(P^k,D) =(|P|-1) \sum_{i} |P|^{k-k_i-1} >0 . \]
Furthermore, since the sum of all eigenvalues of $G_{P^k}(D)$ is $0$, there must be an eigenvalue that is negative. Therefore, $G_{P^k}(D)$ must have at least two distinct eigenvalues. We note that, this part holds true for all undirected simple graphs.

Suppose that $G_{P^k}(D)$ has exactly two eigenvalues. We need to show that $k_{i+1}-k_i = 1$ for $1 \leq i \leq s-1.$ Suppose that this is not the case. Then, we can find $1 \leq i \leq s-1$ such that $k_{i+1}-k_{i} >1.$ We then see that $\chi(P^k, D, k-k_{i+1})=0$ since there is no $j$ such that $k_j = k-(k-k_{i+1})-1 = k_{i+1}-1.$ We then see that 
\[ \lambda_{P^{k-k_{i+1}}}(P^k,D) = (|P|-1) \sum_{j \geq i+1} |P|^{k-k_j-1}.\]
Since $\lambda_{P^{k-k_{i+1}}}(P^k, D)>0$ and $G_{P^k}(D)$ has exactly two eigenvalues, we must have $\lambda_{P^{k-k_{i+1}}}(P^k, D) = \lambda_{P^k}(P^k,D)$. Therefore 
$\sum_{i} |P|^{k-k_i-1} =  \sum_{j \geq i+1} |P|^{k-k_j-1}$, which is impossible. We conclude that $k_{i+1}-k_i = 1$ for all $1 \leq i \leq s-1.$

Finally, we remark that $P^{k-1}$ must be in $D.$ In fact, if this is not the case then by \cref{prop:eigen_zero}, $0$ would be another eigenvalue of $G_{P^k}(D).$

Conversely, if $(k_1, k_2, \ldots, k_s) = (k_1, k_1+1, \ldots, k_1+s-1)$, by \cref{prop:connected_prime_power}, we can see that  $G_{P^k}(D)$ is the disjoint union of $|P|^{k_1}$ copies of $K_{m}$ where $m=|P|^{k-k_1}.$ Therefore, it has exactly two eigenvalues namely $-1$ and $|P|^{k-k_1}-1.$
    \end{proof}

We now conclude this section with the main theorem, which says that $D$ is determined by the spectrum of $G_{P^k}(D).$

\begin{thm}
    Let $D_1, D_2$ be two subsets of $\Div(P^k)$. Then the following statements are equivalent. 

    \begin{enumerate}
        \item $D_1 =D_2.$
        \item $G_{P^k}(D_1)$ and $G_{P^k}(D_2)$ are isomorphic. 
        \item $G_{P^k}(D_1)$ and $G_{P^k}(D_2)$ are isospectral. 
    \end{enumerate}
\end{thm}

\begin{proof}
    By definition $(1) \implies (2) \implies (3).$ Let us show that $(3) \implies (1)$. In fact, suppose that $G_{P^k}(D_1)$ and $G_{P^k}(D_2)$ are isospectral. Then, in particular, they share the same largest eigenvalue. By \cref{cor:largest_eigenvalues}, we must have $D_1 = D_2.$
\end{proof}

\section{Isomorphic gcd-graphs} \label{sec:isomorphic_gcd}
In this section, we provide several constructions of isomorphic gcd-graphs over $\F_q[x]$ that have different generating sets.
\subsection{Isomorphic unitary Cayley graphs in the family $\F_q[x]/f$} 
Let $R$ be a finite commutative ring. The unitary Cayley graph $G_R$ on $R$ is defined as the graph whose vertex set is $R$ and two vertices $a,b$ are adjacent if $a-b \in R^{\times}.$ Let $\Rad(R)$ be the Jacobson radical of $R$ and $R^{\s} = R/\Rad(R)$ be the semi-simplification of $R.$ In \cite[Section 4]{chudnovsky2024prime}, it is shown that $\Rad(R)$ is a homogeneous set in $R$. Furthermore, $G_{R}$ is isomorphic to the wreath product $G_{R^{\s}} * E_m$ where $m = |\Rad(R)|.$ In \cite[Theorem 5.3]{kiani2012unitary}, Kiani and Aghaei show that for two commutative rings $R_1,R_2$, if $G_{R_1} \cong G_{R_2}$, then $R_1^{\s} \cong R_2^{\s}$. Combining these two facts, we have the following.

\begin{prop} \label{prop:isomorphic_criterior}
Suppose that $R_1, R_2$ are two commutative rings. Then the following conditions are equivalent: 
    \begin{enumerate}
        \item $G_{R_1} \cong G_{R_2}$
        \item $|R_1| = |R_2|$ and $R_{1}^{\s} \cong R_2^{\s}.$
    \end{enumerate}
\end{prop}
We will use \cref{prop:isomorphic_criterior} to classify isomorphism classes of unitary graphs in the family $\F_q[x]/f$ where $f$ is a monic polynomial of fixed degree $n.$  We remark that, $G_{\F_q[x]/f}$ is nothing but $G_{f}(\{1 \}).$ Let $\rad(f)$ be the radical of $f$, i.e., the product of all distinct prime factors of $f$. Then the Jacobson radical of $\F_q[x]/f$ is precisely the ideal generated by $\rad(f).$ Consequently, 
\[ (\F_q[x]/f)^{\s} \cong \F_q[x]/\rad(f).\]
By \cref{prop:isomorphic_criterior}, the isomorphism class of $G_{f}(\{1\})$ depends only on $\rad(f)$. We remark, however, that two polynomials with different radicals can still give rise to isomorphic unitary graphs. For instance, take the following examples $f_1 = x^2, f_2 = (x+1)^2.$ Then $\rad(f_1)=x, \rad(f_2) = x+1$, but 
\[ \F_q[x]/\rad(f_1) \cong \F_q[x]/\rad(f_2) \cong \F_q.\]
Consequently, $G_{f_1}(\{1\}) \cong G_{f_2}(\{1\})$ even though $f_1, f_2$ have different radicals. In general, Galois theory for $\F_q$ implies that if $g_1$ and $g_2$ are two irreducible polynomials of the same degree, then $\F_q[x]/g_1 \cong \F_q[x]/g_2.$ Motivated by this observation, we introduce the following definition. 

\begin{definition}
    Let $f \in \F_q[x]$ be a monic polynomial of degree $n$. We define the factorization type of $f$ as the $n$-tuple $(a_1, a_2, \ldots, a_n)$ where  
    $a_i$ is the number of distinct irreducible factors of $f$ of degree $i$. 
\end{definition}  

\begin{expl}
    Let $f = x^2(x+1)(x^2+1) \in \F_3[x]$. Then, the degree of $f$ is $5$ and the factorization type of $f$ is $(2, 1, 0, 0, 0)$.
\end{expl}

By \cref{prop:isomorphic_criterior} and the previous discussion, we have the following. 
 \begin{prop} \label{prop:isomorphic_unitary_criterior}
Let $f_1, f_2 \in \F_q[x]$ be two monic polynomials. The unitary graphs $G_{f_1}(\{1\})$ and $G_{f_2}(\{1\})$  are isomorphic if and only if $\deg(f_1) = \deg(f_2)$ and $f_1, f_2$ have the same factorization type. 
\end{prop}

We will now use \cref{prop:isomorphic_unitary_criterior} to study the number of isomorphism classes of unitary graphs in the family $\F_q[x]/f$ where $f \in \F_q[x]$ and $\deg(f)=n.$ \cref{tab:number_of_iso} shows this number for various values of $n$ and $q$. As we can observe from this dataset, once we fix $n$, the number of isomorphism classes seems to stabilize when $q$ gets bigger. In fact, we have the following proposition. 

\begin{table}[h!]
\centering
\large
\begin{tabular}{|c|c|c|c|c|c|c|c|c|}
\hline
 \diagbox{$n$}{$q$}   & 2 & 3 & 4 & 5 & 7 & 8 & 9 & 11 \\ 
\hline
1  & 1 & 1 & 1 & 1 & 1  & 1  & 1  & 1  \\ 
\hline
2  & 3 & 3 & 3 & 3 & 3  & 3  & 3  & 3  \\ 
\hline
3  & 4 & 5 & 5 & 5 & 5  & 5  & 5  & 5  \\ 
\hline
4  & 7 & 9 & 10 & 10 & 10  & 10  & 10  & 10  \\ 
\hline
5  & 9 & 12 & 13 & 14 & 14  & 14  & 14  & 14  \\ 
\hline
6  & 15 & 22 & 24 & 25 & 26  & 26  & 26  & 26  \\ 
\hline
\end{tabular}
\caption{The number of isomorphism classes of unitary graphs in the family $\F_q[x]/f$ where $\deg(f)=n$ for various values of $n$ and $q$}
\label{tab:number_of_iso}
\end{table}

\begin{prop} \label{prop:number_of_isomorphism}
    Let $n$ be a fixed number. Then, there exist constants $q_n, C_n$ depending only on $n$ such that if $q \geq q_n$, then the number of isomorphism classes of unitary graphs of the form $G_{f}(\{1 \})$ where $f \in \F_q[x]$ and $\deg(f)=n$ is exactly $C_n.$
\end{prop}

In order to provide a proof for \cref{prop:number_of_isomorphism}, we recall that for fixed $m,q$, the number $S(m,q)$ of irreducible polynomials of degree $m$ over $\F_q[x]$ is given by the following Gauss's formula (see \cite{chebolu2011counting} for a proof of this formula using the inclusion-exclusion principle)
\begin{equation} \label{eq:number_of_irr}
S(m,q)=\frac{1}{m} \sum_{d \mid m} \mu(m/d)q^d.
\end{equation}
By \cref{eq:number_of_irr}, we see that for fixed $n$, $S(m,q)$ is asymptotically equivalent to $q^m/m$ for all $m \leq n.$ In particular, $S(m,q) \geq n$ for all $q$ sufficiently large. We can now give a proof for \cref{prop:number_of_isomorphism}

\begin{proof}
    Let $f \in \F_q[x]$ be a monic polynomial of degree $n$ and let $(a_1, a_2, \ldots,a_n)$ be its factorization type. Then we have 
    \[ \sum_{m=1}^n m a_m  \leq n.\]
    In particular, this shows that $a_m \leq n$ for all $1 \leq m \leq n.$ We can find a number $q_n$ such that for $q \geq q_n$, we have $n \leq  S(m,q)$. Consequently, $a_m \leq S(m,q)$ as well.  We then see that, as long as this condition is satisfied, there are sufficiently many irreducible
polynomials of each degree $m$ to realize every factorization type that occurs for a monic polynomial of degree $n$. Thus the collection
of possible factorization types is independent of $q$ for all $q\ge q_n$. 
\end{proof}

\begin{rem}
After the submission of this paper, we completed another paper studying the monotonicity of $S(m,q)$ with respect to both $q$ and $m$; see
\cite{chebolu2026analytic}. In particular, \cite[Theorem 5.5]{chebolu2026analytic} shows that
$ S(m,q)\geq S(1,q)=q$ unless $(m,q)=(2,2)$. Therefore, in \cref{prop:number_of_isomorphism}, we may take $q_n=n$, as suggested by
\cref{tab:number_of_iso}.
\end{rem}

\subsection{Isomorphic gcd-graphs}
In the previous section, we discussed the necessary and sufficient conditions for two unitary Cayley graphs to be isomorphic. In this section, we extend this further to the case of gcd-graphs. The case where $f$ is a prime power is studied exclusively in \cref{sec:prime_powers-II} so we will focus on the case where $f$ has at least two distinct factors in this section.  Since many of our insights arise from analyzing experimental data, we will begin this section with a concrete numerical example.

\begin{expl}
Let $f= x(x+1) \in \F_3[x].$  \cref{table:1} describes the isomorphism classes in the family of $G_{f}(D)$ where $D$ runs over the collection of subsets of $\Div(f)$. We remark that we group these graphs by their characteristic polynomials; i.e. by isospectral classes, but we have checked that elements in the same class are isomorphic as well. This is consistent with a conjecture of So for gcd-graphs over $\Z$ saying that two gcd-graphs $G_{f}(D_1)$ and $G_f(D_2)$ are isomorphic if and only if they are isospectral (see \cite{so2006integral}).

\begin{table}[h!]
\centering
\begin{tabular}{|c|c|}
\hline
\textbf{D} & \textbf{Characteristic polynomial of $G_{f}(D)$} \\
\hline
$\left[ \right]$ & $x^9$ \\
\hline
$ \left[ 1 \right], \left[ x, x + 1 \right] $ & $(x - 4)(x - 1)^4(x + 2)^4$ \\
\hline
$ \left[ x \right], \left[ x + 1 \right]$ & $(x - 2)^3(x + 1)^6$ \\
\hline
$\left[ 1, x \right], \left[ 1, x + 1 \right] $ & $(x - 6)(x + 3)^2x^6$ \\
\hline
$ \left[ 1, x, x + 1 \right] $ & $(x - 8)(x + 1)^8$ \\
\hline
\end{tabular}
\caption{Isomorphism classes in the family $G_{f}(D)$ for $f=x(x+1)$ }
\label{table:1}
\end{table}

Looking closely at the data, we have the following observation. While the isomorphism between $G_{f}(\{x\})$ and $G_{f}(\{x+1\})$ as well as the isomorphism between $G_{f}(\{1, x\})$ and $G_{f}(\{1, x+1\})$ are evident (via a substitution of variables), the isomorphism between $G_{f}(\{1\})$ and $G_{f}(\{x, x+1\})$ is less so. In fact, this isomorphism appears like a coincidence. More precisely, let us consider $f=x(x+1) \in \F_q[x]$ for a generic value of $q$. Then the degree $G_{f}(\{1\})$ is $(q-1)^2$ and the degree of $G_{f}(\{x,x+1\}) = 2(q-1).$
These two degrees are equal only in the case $q=3.$
\end{expl}

We discuss a more complicated example. 
\begin{expl} \label{expl:more_complicated_ex}
Let $f = x^2(x+1) \in \F_3[x].$ As expected, there are more isomorphism classes in this family. We remark, however, that similar to the previous example some isomorphisms depend on the fact that $q=3.$ More precisely, if we let $f = x^2(x+1) \in \F_5[x]$, then we no longer have an isomorphism between $G_{f}(\{1, x, x + 1\})$ and $G_{f}(\{1, x + 1, x^2, x (x + 1)\})$. On the other hand, some isomorphisms persist when we change $q$. For example, using the Python library NetworkX we find the following two phenomena.  

\begin{enumerate}
\item For every $q \leq 5$ and $f = x^2(x+1) \in \F_q[x]$, the graphs $G_{f}(\{1,x,x^2\})$ and \linebreak $G_{f}(\{1, x+1\})$ are isomorphic. 
\item For every $q \leq 5$ and $f= x^2(x+1) \in \F_q[x]$, the graphs $G_{f}(\{1, x+1, x^2\})$ and $G_{f}(\{1, x+1, x(x+1)\})$ are also isomorphic. 
\end{enumerate}

\begin{table}[h!]
\centering
\begin{tabular}{|c|c|}
\hline
\textbf{D} & \textbf{Characteristic polynomial of $G_{f}(D)$} \\
\hline
$\left[ \right]$ & $x^{27}$ \\
\hline
$\left[ 1 \right], \left[ x, x + 1, x^2 \right]$ & $(x - 12)(x - 3)^4(x + 6)^4x^{18}$ \\
\hline
$\left[ x \right], \left[ x^2, x(x + 1) \right] $ & $(x - 4)^3(x - 1)^{12}(x + 2)^{12}$ \\
\hline
$ \left[ 1, x \right] $ & $(x - 16)(x + 8)^2(x + 2)^8(x - 1)^{16}$ \\
\hline
$ \left[ x + 1 \right], \left[ x, x^2 \right], \left[ x, x(x + 1) \right] $ & $(x - 6)^3(x + 3)^6x^{18}$ \\
\hline
$\left[ 1, x + 1 \right], \left[ 1, x, x^2 \right] $ & $(x - 18)(x + 9)^2x^{24}$ \\
\hline
$\left[ x, x + 1 \right] $ & $(x - 10)(x - 4)^2(x + 5)^4(x + 2)^6(x - 1)^{14}$ \\
\hline
$ \left[ 1, x, x + 1 \right], \left[ 1, x + 1, x^2, x(x + 1) \right]$ & $(x - 22)(x + 5)^2(x - 1)^{12}(x + 2)^{12}$ \\
\hline
$ \left[ x^2 \right], \left[ x(x + 1) \right]$ & $(x - 2)^9(x + 1)^{18}$ \\
\hline
$\left[ 1, x^2 \right]$ & $(x - 14)(x + 4)^2(x + 7)^2(x - 2)^{10}(x + 1)^{12}$ \\
\hline
$ \left[ x + 1, x^2 \right]$ & $(x - 8)(x - 5)^2(x + 4)^4(x - 2)^6(x + 1)^{14}$ \\
\hline
$\left[ 1, x + 1, x^2 \right], \left[ 1, x + 1, x(x + 1) \right], \left[ 1, x, x^2, x(x + 1) \right]$ & $(x - 20)(x + 7)^2(x - 2)^6(x + 1)^{18}$ \\
\hline
$\left[ 1, x, x + 1, x^2 \right], \left[ 1, x, x + 1, x(x + 1) \right]$ & $(x - 24)(x + 3)^8x^{18}$ \\
\hline
$\left[ 1, x(x + 1) \right], \left[ x, x + 1, x^2, x(x + 1) \right] $ & $(x - 14)(x - 5)^4(x + 4)^4(x + 1)^{18}$ \\
\hline
$\left[ 1, x, x(x + 1) \right] $ & $(x - 18)(x + 6)^2(x - 3)^4(x + 3)^6x^{14}$ \\
\hline
$\left[ x + 1, x(x + 1) \right], \left[ x, x^2, x(x + 1) \right] $ & $(x - 8)^3(x + 1)^{24}$ \\
\hline
$ \left[ x, x + 1, x(x + 1) \right] $ & $(x - 12)(x - 6)^2(x - 3)^2(x + 3)^{10}x^{12}$ \\
\hline
$\left[ 1, x^2, x(x + 1) \right] $ & $(x - 16)(x + 5)^2(x - 4)^4(x - 1)^6(x + 2)^{14}$ \\
\hline
$ \left[ x + 1, x^2, x(x + 1) \right]$ & $(x - 10)(x - 7)^2(x - 1)^8(x + 2)^{16}$ \\
\hline
$ \left[ 1, x, x + 1, x^2, x(x + 1) \right] $ & $(x - 26)(x + 1)^{26}$ \\
\hline
\end{tabular}
\caption{Isomorphism classes in the family $G_{f}(D)$ for $f=x^2(x+1) \in \F_3[x]$ }
\end{table}
\end{expl}

We tried various other examples and found the following statement which explains the first phenomenon described in \cref{expl:more_complicated_ex}.
\begin{prop} \label{prop:isomorphic_construction}
    Let $f = f_1 f_2$  where  $\gcd(f_1, f_2)=1.$ Assume further that $\rad(f_1)$ and $\rad(f_2)$ have the same factorization type. 
    Let 
    \[ D_1 = \Div(f_1), \quad  D_2 = \Div(f_2)  .\]
    Then $G_{f}(D_1)$ and $G_f(D_2)$ are isomorphic. In fact, they are both isomorphic to the wreath product $G_{\F_q[x]/\rad(f_1)}(\{1\}) * E_{m}$ where 
    $m= \dfrac{|f|}{|\rad(f_1)|}.$
\end{prop}

\begin{proof}
    By the Chinese remainder theorem 
    \[ R = \F_q[x]/(f_1f_2) \cong \F_q[x]/f_1 \times \F_q[x]/f_2.\]
    Under this isomorphism and by the definition of $D_1$, we can see that 
    \[ (x_1, y_1), (x_2, y_2) \in \F_q[x]/f_1 \times \F_q[x]/f_2 \]
    are adjacent in $G_{f}(D_1)$ if and only if $y_1-y_2 \in (\F_q[x]/f_2)^{\times}.$ We then see that $G_{f}(D_1)$ is isomorphic to the wreath product $G_{f_2}(\{1\}) * E_{m_1}$ where $m_1=|f_1|.$ We remark that we can also get this isomorphism by observing that $f_2R$ is a homogeneous set in $G_{f}(D_1)$ and the required isomorphism follows from \cite[Theorem 5.5]{minavc2024gcd}.

    By the result from the previous section, we further have 
    \[ G_{f_2}(\{1\}) \cong G_{\rad(f_2)}(\{1\}) * E_{m_2}, \]
    where $m_2=|f_2|/|rad(f_2)|.$ Since the wreath product is associative, we have 
    \[ G_{f}(D_1) \cong G_{\rad(f_2)}(\{1\})* E_m ,\]
    where $m= \dfrac{|f|}{|\rad(f_2)|}.$ An identical argument and the fact that $\rad(f_1), \rad(f_2)$ have the same factorization type show that $G_{f}(D_2)$ is isomorphic to $G_{\rad(f_1)}(\{1\}) * E_m$ as well. We conclude that $G_{f}(D_1) \cong G_f(D_2).$
\end{proof}
\begin{rem}
    We thank Professor Ki-Bong Nam for some discussion that led to the statement for \cref{prop:isomorphic_construction}. Specifically, Professor Nam suggested us to study generalized Euler numbers which we now recall. Let $A$ be a PID, $n \in A$  and let $m$ be a divisor of $n$. The generalized Euler number $\varphi_{m}(n)$ is the number of elements in the following set 
    \[ U_m(n) = \{ a \in A/n \mid \gcd(a,m) =1 \}.\]
    When $m=n$, $|U_m(n)|$ is precisely the Euler totient function of $n.$ In the setting of \cref{prop:isomorphic_construction}, $S_{D_1} = U_{f_2}(f_1f_2)$ where $S_{D_1}$ is the generating set of the gcd-graph $G_{f}(D_1).$
\end{rem}

We have the following corollary which is a by-product of the proof for \cref{prop:isomorphic_construction}. We first introduce a standard notation. Suppose that $g$ is an irreducible polynomial, we write $g^a \mid \mid f$ if $g^a \mid f$ but $g^{a+1} \nmid f.$

\begin{cor}
Let $f$ be a polynomial. Suppose that $f$ has distinct irreducible factors $f_1, f_2$  of the same degree. Let $a_1, a_2$ be two positive integers such that $f_1^{a_1} \mid \mid f$ and $f_2^{a_2} \mid \mid f.$ Let 
 \[ D_1 = \Div(f_1^{a_1}) , \quad  D_2 = \Div(f_2^{a_2})  .\]
Then $G_{f}(D_1) \cong G_{f}(D_2).$
\end{cor}

\begin{proof}

For $i \in \{1,2 \}$, let $h_i=f/f_i^{a_i}$. By the same  argument as in the proof of \cref{prop:isomorphic_construction} applied to $f = f_i^{a_i} h_i$ we know that 
\[ G_{f}(D_i) \cong G_{\rad(h_i)}(\{1\}) * E_{|f/\rad(h_i)|}.\]
By definition, we know that $\rad(h_1)$ and $\rad(h_2)$ have the same factorization type. Furthermore, we also know that $|\F_q[x]/(f/\rad(h_1))|=  |\F_q[x]/(f/\rad(h_2))|$. Therefore, $G_{f}(D_1) \cong G_{f}(D_2).$
\end{proof}

For the second phenomenon described in \cref{expl:more_complicated_ex}, we found the following generalization. 

\begin{prop}
Let $f_1, f_2$ be two distinct irreducible polynomials of the same degree. Let $n \geq 2$ be a positive integer, $f=f_1^nf_2$, 
\[ D_1 = \{1, f_2, f_1^2, \ldots, f_1^{n} \},\]
and
\[ D_2 = \{1, f_2, f_1 f_2 \}.\]
Then $G_{f}(D_1)$ and $G_{f}(D_2)$ are isomorphic.
\end{prop}

\begin{proof}
    By \cite[Theorem 5.5]{minavc2024gcd}, the ideal $I_{f_1}$ generated by $f_1$ is a homogeneous set in $G_{f}(D_1)$ as well as $G_{f}(D_2).$ Furthermore 
    \[ G_{f}(D_1) \cong G_{f_1}(\{1\}) * G_{f_1^{n-1}f_2}(\{f_1, f_1^2, \ldots, f_1^{n-1}\}).\]
Similarly, we also have 
        \[ G_{f}(D_2) \cong G_{f_1}(\{1\}) * G_{f_1^{n-1}f_2}(\{f_2\}).\]
Therefore, to complete the proof, we only need to show that 
\[ G_{f_1^{n-1}f_2}(\{f_1, f_1^2, \ldots, f_1^{n-1}\}) \cong G_{f_1^{n-1}f_2}(\{f_2\}) .\] 
We observe that   $G_{f_1^{n-1}f_2}(\{f_1, f_1^2, \ldots, f_1^{n-1}\})$ is isomorphic to $|f_1|$ disjoint copies of \linebreak  $G_{f_1^{n-2}f_2}(\{1, f_1, \ldots, f_1^{n-2}\})$ and $G_{f_1^{n-1}f_2}(\{f_2\})$ is isomorphic to $|f_2|$ disjoint copies of $G_{f_1^{n-1}}(\{1 \}).$  By \cref{thm:wreath_product_prime_power} and the proof of \cref{prop:isomorphic_construction} we have 
\[ G_{f_1^{n-2}f_2}(\{1, f_1, \ldots, f_1^{n-2}\}) \cong K_{|f_1|} * E_{|f_1|^{n-2}} \cong G_{f_1^{n-1}}(\{1 \}).\]
Summarizing all these isomorphisms, we conclude that $G_{f}(D_1)$ and $G_{f}(D_2)$ are isomorphic.
\end{proof}

\begin{rem}
A by-product of the above proof is that if $f_1$ and $f_2$ are two distinct irreducible polynomials of the same degree then for each $m \geq 1$
\[ G_{f_1^{m}f_2}(\{1, f_1, \ldots, f_1^{m}\}) \cong K_{|f_1|} * E_{|f_1|^{m}} \cong G_{f_1^{m+1}}(\{1 \}).\]
This isomorphism provides a simple construction of two isomorphic gcd-graphs with different moduli. 
\end{rem}

\begin{rem} \label{rem:same_degree}
A conjecture of So \cite[Conjecture 7.3]{so2006integral} says that if $n$ is an integer and $D_1, D_2$ are two distinct subsets of $\Div(n)$, then the two gcd-graphs over $\Z$$, G_n(D_1)$ and $G_n(D_2)$ are not isomorphic. 

As we have shown, this conjecture is not true if we consider gcd-graphs over $\F_q[x].$ We remark, however, that our constructions of isomorphic gcd-graphs of the form $G_{f}(D)$ where $f$ is fixed are based on a crucial fact that $f$ has two irreducible factors of the same degree. This could happen over $\F_q[x]$ but not over $\Z.$ As a result, we wonder whether So's conjecture still holds for the family $G_{f}(D)$ under the additional assumption that irreducible factors of $f$ have different degrees. 
\end{rem}

\section*{Acknowledgements}
We are grateful to Professor Torsten Sander, whose work on gcd-graphs has been a significant source of inspiration. We also appreciate his insightful remarks and continued encouragement. Our thanks extend to Professor Ki-Bong Nam for his interest and valuable suggestions. Parts of this work are motivated by his ideas on generalized Euler numbers.
Last but not least, we would like to thank the referee, Professor Jon Berrick, and the editor for their suggestions  which improved our exposition.

\end{document}